\documentclass{amsart}
\numberwithin{equation}{section}
\usepackage{amssymb, xcolor, tikz-cd, enumitem}
\usetikzlibrary{patterns}
\usetikzlibrary{calc}
\usepackage{comment}
\usepackage{mathtools}
\usepackage{stmaryrd}
 
\usepackage{tikz} 
\usetikzlibrary{arrows,calc,positioning}
\let\cal\mathcal

\def\Dscr{{\cal D}} 

\def\Fscr{{\cal F}}

\def\Hscr{{\cal H}}

\def\Mscr{{\cal M}}

\def\Oscr{{\cal O}}
\def\Pscr{{\cal P}}

\def\Rscr{{\cal R}}
\def\Sscr{{\cal S}}

 \def\Xscr{{\cal X}}
\def\Yscr{{\cal Y}}
\def\Zscr{{\cal Z}}

\def\gscr{{\mathfrak g}}
\let\blb\mathbb
\def\CC{{\blb C}}

\def\QQ{{\blb Q}}

\def \ZZ{{\blb Z}}

\def \NN{{\blb N}}

\newcommand{\cH}{\mathcal{H}}

\newtheorem{lemma}{Lemma}[section]

\newtheorem{theorem}[lemma]{Theorem}
\newtheorem{corollary}[lemma]{Corollary}

\theoremstyle{definition}

\newtheorem{definition}[lemma]{Definition}

{

}

\theoremstyle{remark}

\newtheorem{remark}[lemma]{Remark}

\newdimen\uboxsep \uboxsep=1ex
\def\uboxn#1{\vtop to 0pt{\hrule height 0pt depth 0pt\vskip\uboxsep
\hbox to 0pt{\hss #1\hss}\vss}}

\def\uboxs#1{\vbox to 0pt{\vss\hbox to 0pt{\hss #1\hss}
\vskip\uboxsep\hrule height 0pt depth 0pt}}

\let\amsamp=&
\begingroup
\catcode`\&=13
\gdef\pampmatrix{%
  \begingroup
  \let&=\amsamp
  \begin{pmatrix}%
}
\gdef\endpampmatrix{\end{pmatrix}\endgroup}
\endgroup

\author{Felix Küng}
\address{Département de Mathématique, Université Libre de Bruxelles, Campus de la
Plaine CP 213, Bld du Triomphe, B-1050 Bruxelles, Belgium}
\email{felix.kung@ulb.be}

\author{Špela Špenko}
\address{Département de Mathématique, Université Libre de Bruxelles, Campus de la
Plaine CP 213, Bld du Triomphe, B-1050 Bruxelles, Belgium}
\email{spela.spenko@ulb.be}

\thanks{The authors were supported by a MIS grant from the National Fund for Scientific Research
(FNRS) and an ARC grant from the Université Libre de Bruxelles.  Part of this paper was written while the
authors were in residence at the Simons Laufer Mathematical Sciences Institute (formerly MSRI) 
in Berkeley, California, during the Spring 2024 semester.}

\title[Comparison of CoHA and KHA]{Comparison of cohomological and K-theoretical Hall algebra}

\begin{document}
%\section{Assumptions}
% \marginpar{\tiny(\v{s}) Not so important, but can we remove this?}
\begin{abstract}
	We give a more conceptual construction of a comparison  algebra morphism from the $K$-theoretical Hall algebra to a twist of the cohomological Hall algebra associated to a symmetric quiver constructed by \cite{KS}, and extend the result to quivers with potential. 
\end{abstract}

\maketitle

\section{Introduction}
Let $Q$ be a quiver with a potential $W$. 
The cohomological Hall algebra (CoHA) $\Hscr^W$ associated to $(Q,W)$  was introduced by Kontsevich and Soibelman \cite{KS}, while the K-theoretical Hall algebra (KHA) $\Rscr^W$ was introduced by P{$\rm\breve{a}$}durariu \cite{Padurariu}. A main motivation for these constructions is that KHAs are positive halves of Yangians while CoHAs are halves of quantum affine algebras \cite{Botta2023,Varagnolo2022,Schiffmann2023}. This analogy is witnessed particularly by KHAs and CoHAs of triple quivers with potential recovering on the nose positive halves of Yangians and quantum affine algebras \cite{Yang2018}. However, for more general symmetric quivers with potential, one expects new quantum group-like objects. 

As there exists a natural comparison morphism of quantum affine algebras and Yangians \cite{Gautam2013}, given by exponentiation of roots on Drinfeld polynomials \cite{Drinfeld1990}, a similar comparison is expected between KHAs and a twisted version of the corresponding COHAs.

In \cite{Lunts2024} this comparison has been realized for a symmetric quiver $Q$ and $W=0$ by a modification of the Chern character to obtain the conjectured algebra morphism $\Rscr\to \hat{\Hscr}^\sigma$,\footnote{We write $\Hscr=\Hscr^0$, $\Rscr=\Rscr^0$.} where $\hat{\Hscr}$ is the completion of $\Hscr$ and $\sigma$ is a Zhang twist. The Chern character was modified in a rather ad hoc way in loc.cit. Here we provide a more conceptual explanation of this modification, related to the square root of Todd classes. This further allows us to extend the morphism to quivers with potentials as seen in
\setcounter{section}{4}
\setcounter{lemma}{7}
\begin{theorem}
	The morphism 
	$$v:\Rscr^W\to \left(\widehat{\Hscr^W},\circ\right)$$
	is a morphism of algebras.
\end{theorem}
\setcounter{section}{1}
\setcounter{lemma}{0}

Using this morphism one can use analogous reasoning to \cite{Lunts2024} to identify finite length modules over KHAs with the corresponding objects over CoHAs.
%In the process we introduce a two-sided Zhang twist which might be of some independent interest. 

\subsection{Acknowledgements}
We thank Tudor P{$\rm\breve{a}$}durariu for his deep mathematical insights and helpful comments on style and Shivang Jindal for very useful and interesting discussions.

\section{Recollections}

\subsection{Quiver varieties}

Let $Q$ be an oriented quiver with vertex set $I$ and  $\alpha_{i,j}\in \ZZ_{\geq 0}$ arrows from $i\in I$ to $j\in J$. {\it We assume that $Q$ is symmetric.} For a fixed dimension vector $\gamma = \left(\gamma_i \right)\in \ZZ^I$ we consider the space 
$$\Mscr_\gamma:=\prod_{{i,j}}\CC^{\alpha_{i,j}\gamma_i \gamma_j}$$
 of $Q$-representations with dimension vector $\gamma$. This carries naturally  an action of the algebraic group $G_\gamma:=\prod_{i\in I} \mathrm{GL}_{\gamma_i}\left(\CC\right)$.

\subsection{Cohomological and K-theoretical Hall algebras}

Fix two dimension vectors $\gamma _1,\gamma _2 \in \ZZ ^I_{\geq 0}$ and put $\gamma =\gamma _1+\gamma _2$. Consider the affine subspace $M_{\gamma _1,\gamma _2}\subset M_\gamma$, which consists of representations for which the standard subspaces $\CC ^{\gamma _1^i}\subset \CC ^{\gamma ^i}$ form a subrepresentation. The subspace $M_{\gamma _1,\gamma _2}$ is preserved by the action of the parabolic subgroup $G_{\gamma _1,\gamma _2}\subset G_\gamma$ which consists of transformations preserving subspaces $\CC ^{\gamma _1^i}\subset \CC ^{\gamma ^i}$, $i\in I$.
We have the natural morphisms of stacks
\begin{equation}\label{eq:morphisms_inducing_multiplication}M_{\gamma _1}/G_{\gamma _1}\times M_{\gamma _2}/G_{\gamma _2}\stackrel{p}{\leftarrow} M_{\gamma _1,\gamma _2}/G_{\gamma _1,\gamma _2}\stackrel{i}{\to }M_{\gamma}/G_{\gamma _1,\gamma _2}\stackrel{\pi}{\to }M_{\gamma}/G_{\gamma }.
\end{equation}
The maps $i$ and $\pi$ are proper. 

We will also look at the stacks where we only act with tori, i.e. $M_\gamma/T_\gamma$ where $T_\gamma$ is the diagonal maximal torus in $G_\gamma$ and denote the corresponding maps by $\tilde{p},\tilde{i}$. Note that $\pi$ is the identity in this case.

\medskip

We denote ${\Hscr}_\gamma:=H^\bullet _{G_{\gamma}}(M_{\gamma},\QQ)$, ${\Rscr}_\gamma:=K_0^{G_{\gamma}}(M_{\gamma},\QQ)$. % and the corresponding maps by $\tilde{p},\tilde{i}$.
Moreover, we denote $\tilde{\Hscr}_\gamma:=H^\bullet _{T_{\gamma}}(M_{\gamma},\QQ)$, $\tilde{\Rscr}_\gamma:=K_0^{T_{\gamma}}(M_{\gamma},\QQ)$. %If we omit the coefficients, we assume them to be $\QQ$. 
Set $\Hscr=\oplus_\gamma \Hscr_\gamma, \Rscr=\oplus_\gamma \Rscr_\gamma$, and analogously for $\tilde{\Hscr},\tilde{\Rscr}$. 

As we will use Chern characters to compare these two Hall algebras we need to pass to the completed cohomological Hall algebra $\widehat{\Hscr}$
%$$ \widehat{\Hscr} := \widehat{\mathrm{CoHA}}\left(Q\right)$$
given by $$\widehat{\Hscr}_\gamma \cong \widehat{H}_{G_\gamma}\left(M_\gamma,\QQ\right):=\prod_{n\in \NN}H^n_{G_\gamma}\left(M_\gamma,\QQ\right).$$

\subsubsection{Product on $\Hscr$}
We define the multiplication
$$m_{\gamma _1,\gamma _2}:\cH _{\gamma _1}\otimes \cH _{\gamma _2}\to \cH _{\gamma}$$
as the composition of the isomorphism
$$p^*:H^\bullet _{G_{\gamma _1}}(M_{\gamma _1},\QQ)\otimes
H^\bullet _{G_{\gamma _2}}(M_{\gamma _2},\QQ)\stackrel{\sim}{\to}H^\bullet _{G_{\gamma _1,\gamma _2}}(M_{\gamma _1,\gamma _2},\QQ)$$
with the push forward maps $i_*$ and $\pi _*$.

We define the multiplication $\tilde{m}_{\gamma_1,\gamma_2}$ on $\tilde{\Hscr}$ as the composition $\tilde{i}_*\tilde{p}^*$. 
%We will later also  multiply it by the inverse of the Euler class of $\iota:BG_{\gamma_1,\gamma_2}\to BG_{\gamma_1+\gamma_2}$, i.e. $
%eu(\gamma_1,\gamma_2):=\iota^*\iota_*(1)$, and denote that multiplication $\tilde{m}^e_{\gamma_1,\gamma_2}$.

We denote the multiplications also by $\cdot_\Hscr,\cdot_{\tilde{\Hscr}}$.

\subsubsection{Product on $\Rscr$}
We define the multiplication 
\[\mu_{\gamma_1,\gamma_2}:\Rscr_{\gamma _1}\times \Rscr _{\gamma_2}\to \Rscr_{\gamma_1+\gamma_2}
\] as the composition of the induced maps
$$\mu _{\gamma_1,\gamma_2}=\pi _* i_* p^*:K_0^{G_{\gamma_1}}(M_{\gamma_1})\otimes K_0^{G_{\gamma _2}}(M_{\gamma_2})\to
K_0^{G_\gamma }(M_\gamma).$$

We define the multiplication $\tilde{\mu}_{\gamma_1,\gamma_2}$ on $\tilde{\Rscr}$ as the composition $\tilde{i}_*\tilde{p}^*$. 
%Later it will be multiplied by the inverse of the Euler class of $\iota:BG_{\gamma_1,\gamma_2}\to BG_{\gamma_1+\gamma_2}$, i.e. $
%eu(\gamma_1,\gamma_2):=\iota^*\iota_*(1)$, that multiplication is denoted by $\tilde{\mu}^e_{\gamma_1,\gamma_2}$.

We denote the multiplications also by $\cdot_\Rscr,\cdot_{\tilde{\Rscr}}$.

\subsection{Square roots of power series}

Following \cite[\S 2.3]{Caldararu2005} there exists for every power series in $c_1,c_2,\cdots$ a unique power series expansion 
$$\sqrt{1+c_1+c_2+\cdots}=1+ \frac{1}{2}c_2+\frac{1}{8}\left(4c_2 -c_1^2\right)+\frac{1}{16}\left(8c_3 -4c_1c_2+c_1^3\right)+\cdots$$
such that 
\begin{align*}
\sqrt{1}&=1\\
\sqrt{\mu}\sqrt{\eta}&=\sqrt{\mu \eta}\\
\sqrt{\mu}^2&=\mu.
\end{align*}
for any elements $\eta,\mu \in \widehat{H}^{even}_G\left(X,\QQ\right)$ with constant term $1$.

\subsection{Grothendieck-Riemann-Roch for stacks}

Similar to \cite{Padurariu} we use the construction of \cite{Krishna2014} to get a Grothendieck-Riemann-Roch Theorem for comparing $G$-theory, 
respectively $K$-theory, of quotient stacks with their Chow rings. 
This construction can then be combined with Totaro's studies of the Chow ring of classifying spaces \cite{Totaro}
to get a Chern character map commuting up to a Todd class with push forwards along proper morphisms $f:X\to Y$,

$$\tikz[xscale=3,yscale=2]{
\node (HX) at (0,0) {$\widehat{H^*_G}\left(X,\CC\right)$};
\node (HY) at (1,0) {$\widehat{H^*_G}\left(Y,\CC\right).$};
\node (KY) at (1,1) {$K_G\left(Y\right)$};
\node (KX) at (0,1) {$K_G\left(X\right)$};

\draw[->]
(KX) edge node[left] {$Td_X\mathrm{ch}$}(HX)
(KX) edge node[above]{$f_*$}(KY)
(KY) edge node[right] {$Td_Y\mathrm{ch}$} (HY)
(HX) edge node[above] {$f_*$} (HY);
}$$

Moreover, the Chern character commutes with pullback. 

\subsection{Todd classes}
The Todd class on $X/G$ can be computed using the definition of $T_{X/G}$ as the complex
\[
\gscr\otimes \Oscr_X \to T_X
\]
with $T_X$ in degree $0$. The action of $G$ on $\gscr$ is the adjoint action. Therefore 
\[
Td_{X/G}=Td_X Td_\gscr^{-1}.
\]

\subsection{Twisted multiplication}\label{sec:twisting_systems}

We refer to \cite{Zhang1996} (or \cite[\S.4]{Lunts2024} for a brief review) for twisted multiplications on an algebra in its utmost generality. 
Here we only focus on a special case that suffices for our application.

\begin{comment}
\begin{definition}
Let $A=\oplus_{g\in \Gamma} A_g$ be an algebra graded by a semi-group $\Gamma$. A family of automorphisms $\left\{\sigma_i: A_g\to A_g\right\}$ is a left twisting system if 
$$\sigma_l\left(\sigma_s \left(x\right)y\right)=\sigma_{sl}\left(x\right)\sigma_l\left(y\right)$$
for all $g,s,l\in \Gamma$ and all $x\in A_g$, $y\in A_s$, 
and a right twisting system if 
$$\sigma_g\left(y\sigma_s\left(z \right)\right)=\sigma_g\left(y\right)\sigma_{gs}\left(z\right)$$
for all $g,s,l\in \Gamma$ and all $y\in A_s$, $z\in A_l$.
\end{definition}

Given a $\Gamma$-graded algebra $A$ and a left twisting system $\left\{\sigma_g \right\}_{g\in \Gamma}$ we can define a new graded algebra $A_\sigma$ with same graded pieces and multiplication given by
$$x \cdot_\sigma y := \sigma_s (x)y$$ for all $x\in A_g$, $y\in A_s$. Similarly we can use a right twisting system to define the algebra ${}_\sigma A$ with multiplication given by
$$x\cdot_{\sigma}y := x \sigma_g(y)$$
for all $x\in A_g$, $y\in A_s$. 

By \cite[Definition~and~Proposition~2.3]{Zhang1996} these define an associative algebra structure.
\end{comment}

Assume that $\Gamma$ is an abelian semigroup and $A$ a $\Gamma$-graded algebra. Let ${\rm Aut}(A)$ denote the group of degree preserving algebra automorphisms of $A$ and $\sigma:\Gamma\to {\rm Aut}(A)$, $g\mapsto \sigma_g$, be a homomorphism of semigroups. Then we can define two new (associative) multiplications on $A$, by setting 
\[
a\circ b=\sigma_r(a)b, \quad a\circ b=a\sigma_l(b)
\]
for all $a\in A_l$, $b\in A_r$.

\begin{comment}
\subsection{Chern character}
The Chern character 
\begin{itemize}
	\item commutes with pullback,
	\item for $f:X\to Y$, \marginpar{\tiny (\v{s}) What are conditions? Proper morphism between smooth schemes? Find reference for stacks.}
	$ch(f_*\Fscr)Td_Y=f_*(ch(\Fscr))f_*Td_X$.
	%\item \marginpar{\tiny (\v{s}) Add reference to Caldararu.} for $f:X\to Y$, let $v$
\end{itemize}
\end{comment}

\section{Comparing multiplication on KHA and CoHA}
\subsection{Todd classes for representation stacks}
We have by \cite[Remark 4.3]{Belmans2024}
\begin{align*}
Td_{M_\gamma/T_\gamma}&=\prod_{i,j\in I}\prod_{\alpha_1=1}^{\gamma^i}
\prod_{\alpha_1\neq \alpha_2=1}^{\gamma^j}\left(\frac{x_{i,\alpha_1}-x_{j,\alpha_2}}{1-e^{x_{j,\alpha_2}-x_{i,\alpha_1}}}\right)^{a_{ij}},\\
Td_{M_\gamma/G_\gamma}&= Td_{M_\gamma/T_\gamma} \prod_{i\in I}\prod_{\alpha_1\neq \alpha_2=1}^{\gamma^i}\frac{1-e^{x_{i,\alpha_2}-x_{i,\alpha_1}}}{x_{i,\alpha_1}-x_{i,\alpha_2}}.
\end{align*}

We denote 
\[
Td_{G_\gamma}:=\prod_{i\in I}\prod_{\alpha_1\neq \alpha_2=1}^{\gamma^i}\frac{x_{i,\alpha_1}-x_{i,\alpha_2}}{1-e^{x_{i,\alpha_2}-x_{i,\alpha_1}}}.
\]

 Let 
\begin{align*}
Td_{\overline{M_{{\gamma_1\gamma_2}}}}(x',x'')&:= \prod_{i,j\in I}\prod_{\alpha_1=1}^{\gamma_1^i}\prod_{\alpha_2=1}^{\gamma_2^j}\left(\frac{x'_{i,\alpha_1}-x''_{j,\alpha_2}}{1-e^{x''_{j,\alpha_2}-x'_{i,\alpha_1}}}\right)^{a_{ij}},\\
Td_{\overline{G_{{\gamma_1\gamma_2}}}}(x',x'')&:=\prod_{i\in I}\prod_{\alpha_1=1}^{\gamma_1^i}\prod_{\alpha_2=1}^{\gamma_2^i}\frac{x'_{i,\alpha_1}-x''_{i,\alpha_2}}{1-e^{x''_{i,\alpha_2}-x'_{i,\alpha_1}}}.
\end{align*}

We denote ${\bf{x}}_{\gamma^i}=\sum_{\alpha=1}^{\gamma^i}x_{i,\alpha}$. 
Let 
\[\tilde{b}_{\gamma_2}^{\gamma_1}:=\prod_{i\in I}exp({\bf{x}}_{\gamma_1^i})^{\sum_{j\in I}a_{ij}\gamma_2^j}, \quad d_{\gamma_2}^{\gamma_1}:=\prod_{i\in I}exp({\bf{x}}_{\gamma_1^i})^{\gamma_2^i}.\] 
Note
\begin{align}\label{Td12vsTd21}
Td_{\overline{M_{{\gamma_1\gamma_2}}}}(x',x'')Td_{\overline{M_{{\gamma_2\gamma_1}}}}(x'',x')^{-1}%=\prod_{i\in I}\frac{exp({\bf{x}}_{\gamma_1^i})^{\sum_{j\in I}a_{ij}\gamma_2^j}}{exp({\bf{x}}_{\gamma_2^i})^{\sum_{j\in I}a_{ij}\gamma_1^j}}
&=\frac{\tilde{b}_{\gamma_2}^{\gamma_1}(x')}{\tilde{b}_{\gamma_1}^{\gamma_2}(x'')},\\\label{TdG12vs21}
Td_{\overline{G_{{\gamma_1\gamma_2}}}}(x',x'')Td_{\overline{G_{{\gamma_2\gamma_1}}}}(x'',x')^{-1}
&=\frac{d_{\gamma_2}^{\gamma_1}(x')}{d_{\gamma_1}^{\gamma_2}(x'')}.
\end{align}

If $\gamma=\gamma_1+\gamma_2$ and $x'_{i,\alpha}=x_{i,\alpha}$, $1\leq \alpha\leq \gamma_1^i$, $x''_{i,\alpha}=x_{i,\gamma_1^i+\alpha}$,  $1\leq \alpha\leq \gamma_2^i$, $i,j\in I$, we get 
\begin{equation}\label{eq:Todd_split}
Td_{M_\gamma/T_\gamma}=Td_{M_{\gamma_1}/T_{\gamma_1}}(x')Td_{M_{\gamma_2}/T_{\gamma_2}}(x'')Td_{\overline{M_{{\gamma_1\gamma_2}}}}(x',x'')Td_{\overline{M_{{\gamma_2\gamma_1}}}}(x'',x'),
\end{equation}
\begin{multline}\label{eq:Todd_split_G}
Td_{M_\gamma/G_\gamma}=Td_{M_{\gamma_1}/G_{\gamma_1}}(x')Td_{M_{\gamma_2}/G_{\gamma_2}}(x'')Td_{\overline{M_{{\gamma_1\gamma_2}}}}(x',x'')Td_{\overline{M_{{\gamma_2\gamma_1}}}}(x'',x')\\
\cdot Td_{\overline{G_{\gamma_1\gamma_2}}}(x',x'')^{-1}Td_{\overline{G_{\gamma_2\gamma_1}}}(x'',x')^{-1}.
\end{multline}
For simplicity, below we will omit the variables, when they are clear from the context.

\subsection{Properties of multiplications and Chern character}
For further reference we single out important properties of the multiplications and Chern character on $\tilde{\Hscr},\tilde{\Rscr}$. 

Let $\Sscr=\oplus_\gamma H_{T_\gamma}^\bullet(\cdot,\QQ)$ and let $q:\cdot/T_\gamma\to \cdot/T_{\gamma_1}\times \cdot/T_{\gamma_2}$ for $\gamma=\gamma_1+\gamma_2$. Then we make $\Sscr$ into an algebra by defining the multiplication via the map induced by $p^*$. 

Note that there is a module action of $\Sscr_{\gamma}$ on $\tilde{\Hscr}_{\gamma}$ induced by the projection $pr:M_{\gamma}/T_{\gamma}\to \cdot/T_{\gamma}$.

%For $f_1\in \Hscr_{\gamma_1} us write $f_1\cdot_\Hscr f_2$, resp. $\cdot_{\tilde{\Hscr}}$, for the multiplication $m_{}
\begin{lemma}\label{lem:mult_prop_T}
Let $f_1\in \tilde{\Hscr}_{\gamma_1}$, $f_2\in \tilde{\Hscr}_{\gamma_2}$, $t_1\in \Sscr_{\gamma_1}$, $t_2\in \Sscr_{\gamma_2}$. Then 
\[
(t_1f_1)\cdot_{\tilde{\Hscr}}(t_2f_2)=(t_1\cdot_{\Sscr}t_2)(f_1\cdot_{\tilde{\Hscr}}f_2).
\]
%The same result holds for $\tilde{\Hscr}$ replaced by $\tilde{\Rscr}$.
\end{lemma}

\begin{proof}
This follows by explicit formulas for multiplication \cite{KS}.\footnote{Note that in loc.cit., the explicit formulas for the multiplication are given in the $G$-equivariant setting, here they are simpler, we do not need shuffle, and we do not divide by $\prod_{i\in I}\prod_{\alpha_1}^{\gamma_1^i}\prod_{\alpha_2=1}^{\gamma_2^i}(x''_{i,\alpha_2}-x'_{i,\alpha_1})$.}%\marginpar{\tiny (\v{s}) Add a reference or find a more intrinsic proof.}

Alternatively, the two sides can be written as follows. The left hand side is equal to 
\[
\tilde{i}_*\tilde{p}^*(pr^* t_1 \cup f_1,pr^* t_2\cup f_2),
\]
while the right-hand side equals
\[
pr^*q^*(t_1,t_2)\cup \tilde{i}_*\tilde{p}^*(f_1,f_2).
\]
Note that $(pr,pr)\circ \tilde{p}=q\circ pr\circ i$. We then have 
\begin{align*}
\tilde{i}_*\tilde{p}^*(pr^* t_1 \cup f_1,pr^* t_2\cup f_2)&=\tilde{i}_*(\tilde{p}^*(pr^*t_1,pr^*t_2)\cup \tilde{p}^*(f_1,f_2))\\
&=\tilde{i}_*(\tilde{i}^*pr^*q^*(t_1,t_2)\cup \tilde{p}^*(f_1,f_2)),
\end{align*}
which equals the right-hand side by the projection formula. 
%\marginpar{\tiny(\v{s}) Please provide reference for the first equality. Need to think about it. I think this equality is crucial when acting on the critical cohomology, $\cup$ is replaced by the action in Ben's paper, S.2.6.}
%where $q:\bullet/G_{\gamma_1+\gamma_2}\to \bullet/G_{\gamma_1}\times \bullet/G_{\gamma_2}$.
\end{proof}

\begin{comment}
\begin{remark}
We may interpret  Lemma \ref{lem:mult_prop_T} as that the module action of $H_{T_\gamma}^\bullet(\cdot)$ on $\tilde{\Hscr}_\gamma$ gives rise to a $\oplus_\gamma H_{T_\gamma}^\bullet(\cdot)$-algebra structure on $\tilde{\Hscr}$.
\end{remark}
\end{comment}

\begin{comment}
\begin{lemma}\label{lem:mult_prop_G}
	Let $t_1,f_1\in {\Hscr}_{\gamma_1}$, $t_2,f_2\in {\Hscr}_{\gamma_2}$. Let $\Pscr(\gamma_1,\gamma_2)$ be the set of shuffles of $(\gamma_1,\gamma_2)$ into $\gamma$. Then 
	\[
	(t_1f_1)\cdot_{{\Hscr}}(t_2f_2)=\left(\sum_{\sigma\in \Pscr(\gamma_1,\gamma_2)}\sigma(t_1t_2)\right)(f_1\cdot_{{\Hscr}}f_2).
	\]
	The same result holds for ${\Hscr}$ replaced by ${\Rscr}$.
	\end{lemma}
	
\end{comment}
%\marginpar{\tiny(\v{s}) Is there an obvious proof?}
%\marginpar{\tiny(f) I think there should be one when we figure out how to interpret $t_1t_2 \in \Hscr_{\gamma_1+\gamma_2}$ which is not yet clear to me.}

\begin{lemma}\label{lem:chern_action}
	The Chern character $ch:\tilde{\Rscr}\to \tilde{\Rscr}$ commutes with the action of $S_n$. 
\end{lemma}

\begin{proof}
Let $f(z_1,...,z_n)\in \tilde{\Rscr}$. Then we have 
$$ ch \left(f\right)=\mathrm{exp}\left(f\left(x_1,...,x_n\right)\right) \in \widehat{{\tilde{\Hscr}}}. $$
In particular we get for $s\in S_n$:
\begin{align*}
ch \left(s f \right)&=\mathrm{exp}\left(s f\left(x_1,...,x_n\right)\right)\\
&=\mathrm{exp}\left( f\left(x_{s1},...,x_{sn}\right)\right)\\
&=\mathrm{exp}\left( f\right)\left(x_{s1},...,x_{sn}\right)\\
&=s\mathrm{exp}\left(f\right)\left(x_1,...,x_n\right)\\
&=s \; ch \left(f\right)
\end{align*}
and so the Chern character is compatible with the action of the symmetric group.
\end{proof}

\subsection{T-equivariant comparison}\label{sec:toric_comparison}
We first work $T$-equivariently. 

Let us denote 
\[\tilde{v}:\tilde{\Rscr}_{\gamma}\to \widehat{\tilde{\Hscr}}_{\gamma},f\mapsto ch(f)Td_{M_\gamma/T_\gamma}^{1/2}.\]
\begin{lemma}\label{lem:mult_Chern}
Let $f_1\in \tilde{\Rscr}_{\gamma_1}$, $f_2\in \tilde{\Rscr}_{\gamma_2}$. We have
\[
\tilde{v}(f_1\cdot_{\tilde{\Rscr}} f_2)=(\tilde{b}_{\gamma_2}^{\gamma_1})^{1/2}\tilde{v}(f_1)\cdot_{\tilde{\Hscr}} (\tilde{b}_{\gamma_1}^{\gamma_2})^{-1/2}\tilde{v}(f_2).
\]
\end{lemma}

\begin{proof}
We use the properties of the Chern character to deduce
\begin{align*}
ch(\tilde{\mu}_{\gamma_1,\gamma_2}(f_1,f_2))&=ch(\tilde{i}_*\tilde{p}^*(f_1,f_2))\\&=Td_{M_\gamma/T_{\gamma}}^{-1}\tilde{i}_*Td_{M_{\gamma_1\gamma_2}/T_\gamma}\tilde{i}_*\tilde{p}^*ch(f_1,f_2)\\
&=Td_{\overline{M_{\gamma_2\gamma_1}}}^{-1}\tilde{m}_{\gamma_1,\gamma_2}(ch(f_1),ch(f_2)).
\end{align*}

 We obtain, using \eqref{eq:Todd_split} and Lemma \ref{lem:mult_prop_T},
\begin{align*}
\tilde{v}(f_1\cdot_{\tilde{\Rscr}} f_2)&=Td_{M_{\gamma_1+\gamma_2}/T_{\gamma_1+\gamma_2}}^{1/2}Td_{\overline{M_{\gamma_2\gamma_1}}}^{-1} ch(f_1)\cdot_{\tilde{\Hscr}} ch(f_2)
\\&=Td_{M_{\gamma_1+\gamma_2}/T_{\gamma_1+\gamma_2}}^{1/2}Td_{\overline{M_{\gamma_2\gamma_1}}}^{-1} Td_{M_{\gamma_1/T_{\gamma_1}}}^{-1/2}Td_{M_{\gamma_2/T_{\gamma_2}}}^{-1/2}\tilde{v}(f_1)\cdot_{\tilde{\Hscr}} \tilde{v}(f_2)
\\&=Td_{\overline{M_{\gamma_2\gamma_1}}}^{-1} Td_{\overline{M_{\gamma_2 \gamma_1}}}^{1/2}Td_{\overline{M_{\gamma_1\gamma_2}}}^{1/2}\tilde{v}(f_1)\cdot_{\tilde{\Hscr}} \tilde{v}(f_2) %&\eqref{eq:Todd_split}
\\&= Td_{\overline{M_{\gamma_2\gamma_1}}}^{-1/2}Td_{\overline{M_{\gamma_1\gamma_2}}}^{1/2} \tilde{v}(f_1)\cdot_{\tilde{\Hscr}} \tilde{v}(f_2) 
\\&=(\tilde{b}_{\gamma_2}^{\gamma_1})^{1/2}(\tilde{b}_{\gamma_1}^{\gamma_2})^{-1/2} \tilde{v}(f_1)\cdot_{\tilde{\Hscr}} \tilde{v}(f_2)
\\&=(\tilde{b}_{\gamma_2}^{\gamma_1})^{1/2}\tilde{v}(f_1)\cdot_{\tilde{\Hscr}} (\tilde{b}_{\gamma_1}^{\gamma_2})^{-1/2}\tilde{v}(f_2).
\end{align*}
\end{proof}

Denote $\tilde{c}_{\tau}^{\gamma}=(\tilde{b}_{\tau}^{\gamma})^{1/2}$. Then
\begin{align*}
\tilde{c}_{\tau}^{\gamma_1}({\bf x'})\tilde{c}_{\tau}^{\gamma_2}({\bf x''})&=
\tilde{c}_{\tau}^{\gamma_1+\gamma_2}({\bf x'},{\bf x''}),\\\nonumber
\tilde{c}_{\tau_1}^\gamma \tilde{c}_{\tau_2}^\gamma&=\tilde{c}_{\tau_1+\tau_2}^\gamma.
\end{align*}
The same formulas hold for $(\tilde{c}_\tau^\gamma)^{-1}$. These implies that the map $\tau\mapsto (f_\gamma\mapsto \tilde{c}_\tau^\gamma f_\gamma)_{\gamma \in \ZZ^I}$ defines a homomorphism of semi-groups $\ZZ^I\to {\rm Aut}(\tilde{\Hscr})$. 

Hence we can define a new (associative) product on $\tilde{\Hscr}$ (see \S\ref{sec:twisting_systems}); i.e. for $f_i\in \tilde{\Hscr}_{\gamma_i}$, $i=1,2$, we set
\[
f_1\circ f_2:=\tilde{c}_{\gamma_2}^{\gamma_1} f_1\cdot_{\tilde{\Hscr}} (\tilde{c}_{\gamma_1}^{\gamma_2})^{-1} f_2.
\]

\begin{corollary}
The map $\tilde{v}$ is an injective morphism of rings from $(\tilde{\Rscr},\mu)$ to $(\widehat{\tilde{\Hscr}},\circ)$. 
\end{corollary}

\begin{proof}
The corollary follows immediately from Lemma \ref{lem:mult_Chern} and the definition of $\circ$. 
\end{proof}

\subsection{G-equivariant comparison}\label{sec:G_comparison}
In this section we establish the algebra comparison of the CoHa and KHA.

Let 
\begin{align*}
e^\Hscr_{\gamma}&=\prod_{i\in I}\prod_{\alpha=1}^{\gamma^i}\prod_{\alpha_1\neq \alpha_2=1}^{\gamma^i}{(x_{i,\alpha_2}-x_{i,\alpha_1})},\\
e^\Rscr_{\gamma}&=\prod_{i\in I}\prod_{\alpha=1}^{\gamma^i}\prod_{\alpha_1\neq \alpha_2=1}^{\gamma^i}{(1-x_{i,\alpha_1}x_{i,\alpha_2}^{-1})},
\end{align*}
and analogously,
\begin{align*}
e^\Hscr_{\gamma_1\gamma_2}&=\prod_{i\in I}\prod_{\alpha_1=1}^{\gamma_1^i}\prod_{\alpha_2=1+\gamma_1^i}^{\gamma_1^i+\gamma_2^i}{(x_{i,\alpha_2}-x_{i,\alpha_1})},\\
e^\Rscr_{\gamma_1\gamma_2}&=\prod_{i\in I}\prod_{\alpha_1=1}^{\gamma_1^i}\prod_{\alpha_2=1+\gamma_1^i}^{\gamma_1^i+\gamma_2^i}{(1-x_{i,\alpha_1}x_{i,\alpha_2}^{-1})}.
\end{align*}

We first record here an easy lemma for further reference that follows directly from definitions. 
\begin{lemma}\label{lem:eu_TdG}
We have 
\[
\frac{e^\Hscr_{\gamma_1\gamma_2}}{ch(e^\Rscr_{\gamma_1\gamma_2})}=Td_{\overline{G_{\gamma_2\gamma_1}}}.
\]
\end{lemma}

Let us denote $\cdot_{\tilde{\Hscr}_e}:=(e^\Hscr)^{-1}\cdot_{\tilde{\Hscr}}$ and $\cdot_{\tilde{\Rscr}_e}:=(e^\Rscr)^{-1}\cdot_{\tilde{\Rscr}}$. We write 
\[v_e:\tilde{\Rscr}_\gamma\to \widehat{\tilde{\Hscr}_\gamma},\quad f\mapsto\tilde{v}(f)Td^{-1/2}_{G_\gamma}=ch(f)Td^{1/2}_{M_\gamma/G_\gamma}.
\]

\begin{lemma}\label{lem:compare_ve}
Let $f_i\in \tilde{\Hscr}_{\gamma_i}$, $i=1,2$. Then
\[
v_e(f_1\cdot_{\tilde{\Rscr}_e}f_2)=(\tilde{b}_{\gamma_2}^{\gamma_1}/d_{\gamma_2}^{\gamma_1})^{1/2}v_e(f_1)\cdot_{\tilde{\Hscr_e}}(\tilde{b}_{\gamma_1}^{\gamma_2}/d_{\gamma_1}^{\gamma_2})^{-1/2}v_e(f_2).
\]
\end{lemma}

\begin{proof}
The proof follows a sequence of easy steps. The first equality below uses the definition of $v_e$, the second  Lemma \ref{lem:mult_Chern}, the third uses the definition of $v_e$ together with \eqref{eq:Todd_split_G} while the forth follows from  Lemma \ref{lem:eu_TdG} and \eqref{TdG12vs21}. 

We have 
\begin{align*}
v_e(f_1\cdot_{\Rscr_e}f_2)
&=ch((e^\Rscr_{\gamma_1+\gamma_2})^{-1})Td_{G_{\gamma_1+\gamma_2}}^{{-1/2}}\tilde{v}(f_1\cdot_{\tilde\Rscr} f_2)\\
&=ch((e^\Rscr_{\gamma_1+\gamma_2})^{-1})Td_{G_{\gamma_1+\gamma_2}}^{{-1/2}}(\tilde{b}_{\gamma_2}^{\gamma_1})^{1/2}\tilde{v}(f_1)\cdot_{\tilde{\Hscr}} (\tilde{b}_{\gamma_1}^{\gamma_2})^{-1/2}\tilde{v}(f_2)\\
&=ch((e^\Rscr_{\gamma_1+\gamma_2})^{-1})Td_{\overline{G_{\gamma_1\gamma_2}}}^{-1/2}Td_{\overline{G_{\gamma_2\gamma_1}}}^{-1/2}e^\Hscr_{\gamma_1+\gamma_2} (\tilde{b}_{\gamma_2}^{\gamma_1})^{1/2}v_e(f_1)\cdot_{\tilde{\Hscr}_e} (\tilde{b}_{\gamma_1}^{\gamma_2})^{-1/2}v_e(f_2)\\
&=(\tilde{b}_{\gamma_2}^{\gamma_1}/d_{\gamma_2}^{\gamma_1})^{1/2}v_e(f_1)\cdot_{\tilde{\Hscr}_e} (\tilde{b}_{\gamma_1}^{\gamma_2}/d_{\gamma_1}^{\gamma_2})^{-1/2}v_e(f_2).\qedhere
\end{align*}
\end{proof}

Let ${c}_{\tau}^{\gamma}=(\tilde{b}_{\tau}^{\gamma}/d_{\tau}^\gamma)^{1/2}$. Then we define a new product on $\Hscr$ (similarly as in \S\ref{sec:toric_comparison}) by setting 
\[
f_1\circ f_2:= c_{\gamma_2}^{\gamma_1}f_1\cdot_\Hscr (c_{\gamma_1}^{\gamma_2})^{-1}f_2
\]
for $f_i\in \Hscr_{\gamma_i}$, $i=1,2$.

Let 
\[v:\Rscr_\gamma\to \hat{\Hscr}_\gamma,\quad f\mapsto Td_{M_\gamma/G_\gamma}^{1/2}ch(f).\]
In order for the map $v:\Rscr_\gamma\to \Hscr_\gamma$ to be well-defined we need an easy lemma.

\begin{lemma}
We have $Td_{M_\gamma/G_\gamma}^{1/2}\in \hat{\Hscr}_\gamma$.
\end{lemma}

\begin{proof}
We need to guarantee that $Td_{M_\gamma/G_\gamma}^{1/2}$ is $S_\gamma$-invariant. Since the leading coefficient of the power series $Td_{M_\gamma/G_\gamma}$ is $1$ and the terms in the expression for $()^{1/2}$ are rational functions in the terms of  $Td_{M_\gamma/G_\gamma}$ that are $S_\gamma$-invariant, it follows that $Td_{M_\gamma/G_\gamma}^{1/2}$ is $S_\gamma$-invariant. 
\end{proof}

\begin{corollary}\label{cor:v-comparison}
The map $v$ is an injective morphism of rings from $(\Rscr,\cdot_\Rscr)$ to $(\hat{\Hscr},\circ)$. 
\end{corollary}

\begin{proof}
Following \cite[\S4.2, Corollary 4.8]{Davison}, \cite[\S3.2, Proposition 3.6]{Padurariu}, $\Hscr$, $\Rscr$ are isomorphic to the $S_n$-invariant part of $\tilde{\Hscr}$, $\tilde{\Rscr}$. This isomorphisms are induced by pullbacks of stacks and a symmetrization which by \cite[Proposition 4.3]{Davison} can be realized as a pullback along a Galois cover. In particular these morphisms are compatible with the Chern character as pullbacks along a morphism.
Using this isomorphism we may view $v$ as the restriction of $v_e$ to the invariant part of $\tilde{\Rscr}$. %\marginpar{\tiny(\v{s}) I guess we need compatibility of the Chern character with the isomorphisms. \tiny{(f)} I checked, both Tudor and Ben use a symmetrization map and pullbacks along geometric comparison maps. The geometric ones we are canonically compatible with and Ben uses a geometric construction for the symmetrization that make it compatible as well}

Multiplication on $\tilde{\Rscr}^{S_n}$, resp. $\tilde{\Hscr}^{S_n}$, is defined as $$\sum_{\sigma\in \Pscr(\gamma_1,\gamma_2)}\sigma \cdot_{\tilde{\Rscr}_e} \text{, resp. }\sum_{\sigma\in \Pscr(\gamma_1,\gamma_2)}\sigma \cdot_{\tilde{\Hscr}_e},$$ where $\Pscr(\gamma_1,\gamma_2)$ is the set of shuffles of $(\gamma_1,\gamma_2)$ in $\gamma$. These are $\tilde{m}$, resp. $m_T$, in \cite[p.35 (arXiv)]{Padurariu}, resp. \cite[p.33 (arXiv)]{Davison}.

Since the Chern character is $S_n$-equivariant (see Lemma \ref{lem:chern_action}) we obtain the desired conclusion by Lemma \ref{lem:compare_ve}.

\end{proof}

\section{Quivers with Potential}

We will now extend our results from the previous section to the case of the K-theoretic Hall algebra and critical cohomological Hall algebras of quivers with potential. In particular we will consider from now on a quiver $Q$ with potential $W$ such that $0$ is the only critical value of the regular function $\mathrm{tr}\,W:M_\gamma/G_\gamma \to \CC$. We denote the critical fibre over $0$ by $M_{\gamma,0}$ and its inclusion by $\iota: M_0\hookrightarrow M$.

In order to use our results in full generality we recall the following result first proven in \cite[Theorem~3.9]{Orlov2004} in its stronger incarnation applying to Landau-Ginzburg models. 
\begin{theorem}\cite[Proposition 3.14]{ballard2014category}\cite[Theorem 3.6]{hirano2017derived}
Let $G$ be a reductive algebraic group acting on a smooth variety $X$ and $W:X \to \CC$ an invariant regular function. Then we have equivalences
$$\mathrm{MF}\left(X/G,W\right)\cong\mathrm{MF}\left(X,W\right)_G\cong \Dscr_G^{sg} \left(X_0\right)\cong \Dscr^{sg}\left(X_0/G\right).$$
\end{theorem}

\subsection{K-theoretic Hall algebra of quiver with potential}

By \cite[Proposition~3.4, Proposition~3.5]{Padurariu} the morphisms from \eqref{eq:morphisms_inducing_multiplication} for $\gamma=\gamma_1 + \gamma_2$ induce exact functors between singularity categories
\begin{align*}\Dscr^{sg}\left(M_{\gamma_1,0}/G_{\gamma_1}\right)\times \Dscr^{sg}\left(M_{\gamma_2,0}/G_{\gamma_2}\right)&\xrightarrow{p^*}\Dscr^{sg} \left(M_{\gamma_1,\gamma_2,0}/G_{\gamma_1,\gamma_2}\right)\\
\Dscr^{sg} \left(M_{\gamma_1,\gamma_2,0}/G_{\gamma_1,\gamma_2}\right)&\xrightarrow{i_*}  \Dscr^{sg}\left(M_{\gamma_1,\gamma_2,0}/G_{\gamma}\right)\\
  \Dscr^{sg}\left(M_{\gamma_1,\gamma_2,0}/G_{\gamma}\right)&\xrightarrow{\pi_*} \Dscr^{sg}\left(M_{\gamma,0}/G_{\gamma}\right).  
  \end{align*}

By \cite[Section~3]{Padurariu} this gives rise to a well-defined associative algebra structure.

\begin{definition}\cite[Definition~3.1]{Padurariu}
The $K$-theoretic Hall algebra of $\left(Q,W\right)$ is given by the $\ZZ^I$-graded abelian group with the $\gamma\in \ZZ^I$ part
$$\mathrm{KHA}\left(Q,W\right)_\gamma:= K_0\left(\Dscr^{sg}(M_{\gamma,0}/G_\gamma\right))=K_0^{G_\gamma}(\Dscr^{sg}(M_{\gamma,0}/G_\gamma))$$
with multiplication given by 
$$f\cdot g:=\pi_*\circ i_* \circ p^* \left(f,g\right).$$
\end{definition}

Observe that for the above result \cite{Padurariu} proves associativity on the categorical level, in particular one can define similar topological $K$-theoretical Hall algebras.

\subsection{Critical cohomological Hall algebra of a quiver with potential}

Similarly \cite[\S 3]{Davison} defines the critical cohomological Hall algebra as the induced algebra on the critical cohomology of $W$. We first recall the definition of the vanishing cycles $\varphi_W$.
\begin{definition}
Let $f:X\to \CC$ be a regular function. 
Consider the inclusion of the zero fibre $\iota: X_0 \hookrightarrow X$ and the pullback diagram 
$$\tikz[xscale=2,yscale=2]{
\node (X) at (0,0) {${X}$};
\node (XT) at (0,1) {$\widehat{X}$};
\node (C1) at (1,0) {$\CC$};
\node (C2) at (1,1) {$\CC$};

\draw[->]
(XT) edge node[left] {$p$}(X)
(XT) edge node[above]{$\widehat{f}$}(C2)
(X) edge  node[above] {$f$}(C1)
(C2) edge node[right] {$\mathrm{exp}$} (C1);
}.$$
Then the vanishing sheaf of $f$ is defined as
$$\iota^* \mathrm{RHom}\left(f^*\left(\mathrm{exp}_!\QQ \to \QQ\right),\QQ\right).$$
\end{definition}
In particular $\varphi_f$ allows the cohomological study of the singularity of $f$ at $0$. We use the shorthand $\varphi_W$ for $\varphi_{\mathrm{tr} W}$.

\begin{definition}
The critical cohomological Hall algebra of $\left(Q,W\right)$ is the $\ZZ^I$-graded algebra given by $\QQ$-vector spaces
$$\mathrm{CoHA}\left(Q,W\right)_\gamma:=H_{G_\gamma}^*\left(M_{\gamma},\varphi_W\right) $$
with multiplication
$$f\cdot g := \pi_*\circ i_* \circ p^*\left(f,g\right).$$
\end{definition}

\subsection{Comparison for quiver with potential}

We now prove that our comparison morphism from Section~\ref{sec:G_comparison} can be carried over to the case of quivers with potential using the algebra morphism $\iota^*: H^*\left(\Xscr\right)\to H^*\left(\Xscr_0\right)$.

For the remainder we fix a quiver with potential $\left(Q,W\right)$ and use the following shorthand for its $K$-theoretic Hall algebra, respectively its cohomological Hall algebra,
\begin{align*}
\Rscr^W  &:=\mathrm{KHA}\left(Q,W\right), \\
 \Hscr^W &:= \mathrm{CoHA}\left(Q, W\right).
\end{align*}

As we will use Chern characters to compare these two Hall algebras we need to pass to the completed cohomological Hall algebra
$$ \widehat{\Hscr}^W := \widehat{\mathrm{CoHA}}\left(Q, W\right)$$
given by $$\widehat{\Hscr}^W_\gamma \cong \widehat{H}_{G_\gamma}\left(M_\gamma,\varphi_W\right):=\prod_{n\in \NN}H^n_{G_\gamma}\left(M_\gamma,\varphi_W\right)$$

%%\marginpar{\tiny(\v{s}) I think the first two are not entirely correct, one would need $K_i$ of $\Dscr^{\rm sg}(M_{\gamma,0}/G_\gamma)$.}
%where the last one is an exact sequence as $H_{G_\gamma}^*\left(M_\gamma \right)$-modules. The action of $H^*_{G_\gamma}\left(M_\gamma \right)$ on $H_{G_\gamma}^{*}\left(M_\gamma,\varphi_W\right)$ factors naturally over $H^*_{G_\gamma}\left(M_{\gamma,0} \right)$ as $\varphi_W$ is supported on $M_{\gamma,0}$.

Let $h: \Xscr \to \Yscr$ be a morphism of stacks, then we have the following version of Grothendieck-Riemann-Roch result arising by composing the map in \cite{Padurariu2023} with the natural map sending monodromy invariant vanishing cycles to all vanishing cycles induced by the splitting of \cite[(6.1)]{Padurariu2023} given in \cite[Lemma~7.2]{Padurariu2023}.
\begin{theorem}\cite[Theorem~6.8]{Padurariu2023}
Let $h:\Xscr \to \Yscr $ be a morphism of smooth quotient stacks, $W_\Yscr:\Yscr \to \CC$ a regular function and set $W_\Xscr:=W_\Yscr \circ h$. 
\begin{enumerate}
\item The following diagram commutes
$$\tikz[xscale=5,yscale=2]{
\node (KMFY) at (0,1) {$K_0\left(\mathrm{MF}\left(\Yscr,W_\Xscr\right)\right)$};
\node (KMFX) at (1,1) {$K_0\left(\mathrm{MF}\left(\Xscr,W_\Yscr\right)\right)$};
\node (HcritY) at (0,0) {$\widehat{H}^*\left(\Yscr,\varphi_{W_\Xscr}\right)$};
\node (HcritX) at (1,0) {$\widehat{H}^*\left(\Xscr,\varphi_{W_\Yscr}\right)$};

\draw[->]
(KMFY) edge node[left] {$ch$}(HcritY)
(KMFX) edge node[left]{$ch$}(HcritX)
(KMFY) edge  node[above] {$h^*$}(KMFX)
(HcritY) edge node[above] {$h^*$} (HcritX);
}.$$
\item Assume that $h$ is proper. Let $Td_h\in \widehat{H}\left(\Xscr_0\right)$ be the Todd class of the virtual tangent bundle $T_h$. Then the following diagram commutes:
$$\tikz[xscale=5,yscale=2]{
\node (KMFX) at (0,1) {$K_0\left(\mathrm{MF}\left(\Xscr,W_\Xscr\right)\right)$};
\node (KMFY) at (1,1) {$K_0\left(\mathrm{MF}\left(\Yscr,W_\Yscr\right)\right)$};
\node (HcritX) at (0,0) {$\widehat{H}^*\left(\Xscr,\varphi_{W_\Xscr}\right)$};
\node (HcritY) at (1,0) {$\widehat{H}^*\left(\Yscr,\varphi_{W_\Yscr}\right)$};

\draw[->]
(KMFX) edge node[left] {$ch$}(HcritX)
(KMFY) edge node[left]{$ch$}(HcritY)
(KMFX) edge  node[above] {$h_*$}(KMFY)
(HcritX) edge node[above] {$Td_h h_*$} (HcritY);
}.$$

\end{enumerate}
%\marginpar{\tiny(\v{s}) Need to define $\phi_W^{inv}$? Do we need this above? (f) can we do this instead in order to avoid a full section on the difference of monodromy invariant vanishing cycles}
\end{theorem}

\begin{remark}
The above theorem is proven in \cite{Padurariu2023} for topological $K$-theory by combining classical Grothendieck-Riemann-Roch with the long exact sequences induced by the short exact sequences, $\Zscr\in \{\Xscr,\Yscr\}$,
$$\Dscr^b \mathrm{Perf}\left(\Zscr_0\right)\hookrightarrow \Dscr^b\left(\Zscr_0\right) \twoheadrightarrow \Dscr^\mathrm{sg}\left( \Zscr_0\right),$$
in particular the same proof works for algebraic $K_0$.
%\marginpar{\tiny(\v{s}) I am really sorry.  It is a bit confusing, since the Theorem 4.5 is in the general setting, while the exact sequences are in our special setting. Perhaps it would be good to make a separate subsection with GRR? We also need to define $\tilde{H}$.}
\end{remark}

%Furthermore, as the proof in loc. cit. arises as completion of GRR for stacks along long exact sequences we have that multiplication with $Td_{\Xscr_0}\in H^*\left(\Xscr_0\right)$ is identical to the multiplication of $Td_\Xscr \in H^*\left(\Xscr\right)$ via restriction.
%The following lemma allows us to apply our computations of Todd classes from Section~\ref{sec:G_comparison} and Section~\ref{sec:toric_comparison} apply to the setting of quivers with potential via the algebra morphism $\iota^*: H^*\left(\Xscr\right)\to H^*\left(\Xscr_0\right)$.

\begin{lemma}\label{lem:Todd class of singularity category is restriction}
We have $\iota^*\circ Td_{h}=Td_{h|_{\Xscr_0}}\in H^*\left(\Xscr_0 \right)$.%\marginpar{\tiny (\v{s}) Is $Td(T_h)=Td_h$?}
\begin{proof}
Consider for a morphism $h:\Xscr \to \Yscr$ of stacks with potential $f_\Yscr:\Yscr \to \CC$ and $f_\Xscr:= f_\Yscr \circ h$ the induced diagram 
$$\tikz[xscale=5,yscale=2]{
\node (X) at (0,1) {$\Xscr$};
\node (Y) at (1,1) {$\Yscr$};
\node (X0) at (0,0) {$\Xscr_0$};
\node (Y0) at (1,0) {$\Yscr_0$};

\draw[->]
(X) edge node[above] {$h$}(Y)
(X0) edge node[above]{$h|_{\Xscr_0}$}(Y0);
\draw[right hook ->]
(X0) edge node[right] {$\iota$}  (X)
(Y0) edge node[left] {$\iota$} (Y);
}.$$
Then we get by Grothendieck-Riemann-Roch $\iota^*\circ Td_{h}=Td_{h|_{\Xscr_0}} \circ \iota^*$. As $\iota^*$ is an algebra morphism we can apply it to $1$ in order to get $\iota^* Td_{h}=Td_{h|_{\Xscr_0}}$ in  $H^*\left(\Xscr_0 \right)$.
\end{proof}
\end{lemma}
Using the above lemma and the algebra morphism $\iota^*$ we can define an analogous comparison map and twisted multiplication to the one from Section~\ref{sec:G_comparison}:
\begin{align*}
v_{sg}:\Rscr^W &\to \widehat{\Hscr^W}\\
f&\mapsto  Td_{M_{\gamma,0}/G_\gamma}^{1/2}ch\left(f\right)
\end{align*}
and a twisted multiplication on $\Hscr^W$ given by
$$f_1\circ f_2 := \left(\iota^*\widehat{c}_{\gamma_2}^{\gamma_1}\right)f_1\cdot_{\Hscr^W}\left(\iota^*\widehat{c}_{\gamma_1}^{\gamma_2}\right)^{-1} f_2.$$

\begin{theorem}\label{thrm: comparison by twisted chern character}
	The morphism 
	$$v:\Rscr^W\to \left(\widehat{\Hscr^W},\circ\right)$$
	is a morphism of algebras.
\end{theorem}

\begin{proof}
	We want to apply our computations from \S\ref{sec:G_comparison} to the case of quivers with potential. 
	We basically repeat the proof,  first passing through the $T$-equivariant setting. We only mention the necessary modifications. 
	
	By Lemma~\ref{lem:Todd class of singularity category is restriction} we have $\iota^*\circ Td_{h}=Td_{h|_{\Xscr_0}}$. This allows us to use formulas \eqref{Td12vsTd21}, \eqref{TdG12vs21}. 
	The only other thing that we need is an analogue of Lemma \ref{lem:mult_prop_T}. The action of $H_{T_\gamma}^\bullet(M_\gamma,\QQ)$ (or equivalently $H_{T_\gamma}^\bullet(\cdot,\QQ)$ as $M_\gamma$ is $T_\gamma$-contractible) on $H^\bullet_{T_\gamma}(M_\gamma,\phi_W)=H^\bullet_{c,T_\gamma}(M_\gamma,\phi_W)^\vee$ is constructed in \cite[\S2.6]{Davison}. To apply the proof of Lemma \ref{lem:mult_prop_T} we need that the action commutes with $\tilde{p}^*$ and $\tilde{i}_*$. 
	
	The action of $H_G^*(\cdot,\QQ)$ on $H_{c,G}(X,\phi_f)^\vee$ for $G<GL_n(\CC)$ in loc.cit. is defined as the limit (over $N\geq n$)  of the pullback maps (cf. the beginning of \S2.7 in loc.cit.)
	\[\Delta_N:(X\times_G fr(n,N))\to 
	(X\times_G fr(n,N))\times (pt\times_G fr(n,N)),\; (x,z)\mapsto ((x,z),(pt,z))
	\]
	on the dual of compactly supported cohomologies with coefficients $\phi_{f_N}$, resp. $\phi_{g_N}$ where $f_N$, resp. $g_N$, is the map induced by $f$ on the domain $A_X$, resp. the target $B_X$, of $\Delta_N$, precomposed by the Thom-Sebastiani isomorphism. 
	
	Then it follows that the action commutes with the pullback. (Note that in our case it is important that the domain and codomain of $\tilde{p}$ are acted upon the same group $T_{\gamma_1}\times T_{\gamma_2}$.)
	
	We next claim that the action commutes with the pushforward of a closed embedding $i:X\to Y$, where both are acted upon $G$. We have the following diagram 
\[
\begin{tikzcd}
A_X\ar[r,"i"]\ar[d,"\Delta^X_N"]&A_Y\ar[d,"\Delta_N^Y"]\\
B_X\ar[r,"i"]&B_Y.
\end{tikzcd}
\]
Note that the images of $B_X$ and $A_Y$ in $B_Y$ intersect transversally. We may then apply \cite[Corollary 2.15]{Davison} which implies that $i_*(\Delta^X_N)^*=(\Delta^Y_N)i_*$. Passing to the limit we obtain that the action commutes with $i_*$. Consequently, in our setting the action commutes with $\tilde{i}_*$, where it is again important that the domain and codomain of $\tilde{i}$ are acted upon the same group $T_{\gamma_1}\times T_{\gamma_2}$.
	%We may combine this observation with Lemma~\ref{lem:compare_ve} and the computations from Corollary~\ref{cor:v-comparison} to get
	\begin{comment}
	\begin{align*}
	v\left(f_1\cdot_{{\Rscr}^W}f_2\right)=& Td_{M_{\gamma,0}/G_\gamma}^{1/2}ch\left(f_1\cdot_{{\Rscr}^W}f_2\right)
	&\text{(definition)}\\
	=&\iota^* Td_{M_{\gamma}/G_\gamma}^{1/2}ch\left(f_1\cdot_{{\Rscr}^W}f_2\right)&\text{(Lemma~\ref{lem:Todd class of singularity category is restriction})}\\
	=&\left(\iota^* \widehat{c}_{\gamma_2}^{\gamma_1}\right)v(f_1)\cdot_{\widehat{\Hscr_e}}\left(\iota^*\widehat{c}_{\gamma_1}^{\gamma_2}\right)v\left(f_2\right)&\text{(Lemma~\ref{lem:compare_ve})}\\
	=&v\left(f_1\right)\circ v \left(f_2\right).
	\end{align*}
	\end{comment}
	%\marginpar{\tiny (\v{s}) I do not understand the forelast equality.}
	%So $v$ is compatible with the twisted multiplication and hence induces a morphism of algebras.
	\end{proof}

\begin{remark}
As \cite{Gautam2013} constructed a comparison of the finite length modules over Yangians and quantum affine algebras a similar result is expected for CoHAs and KHAs. This holds as the arguments from \cite{Lunts2024} carry over for the morphism constructed in Theorem~\ref{thrm: comparison by twisted chern character}. In particular one gets an identification of finite length modules over $\Rscr$ and $\Hscr$.
\end{remark}

\bibliographystyle{alpha}
\bibliography{chern}  

\end{document}